\documentclass[leqno]{amsart}
\usepackage{amscd,amssymb,amsmath}

\usepackage{color}

\usepackage{amsmath,amsfonts,amsthm,amssymb}
\usepackage[all]{xy}     %para usar o XYPic
\numberwithin{equation}{section}
\theoremstyle{definition}
\newtheorem{theorem}[equation]{Theorem}

\newtheorem{remark}[equation]{Remark}

\newtheorem{corollary}[equation]{Corollary}

\newtheorem{definition}[equation]{Definition}

\newtheorem{lemma}[equation]{Lemma}
\newcommand{\Z}{\mathbb{Z}}
\begin{document}

\title[Cobordism of maps on $\Z_2$-Witt spaces]{Cobordism of maps on $\Z_2$-Witt spaces}
\thanks {
The first named author is partially supported by  USP-COFECUB and by the
Projeto Pesquisador Visitante Especial do Programa Ci\^encia sem Fronteiras,
Processo 400580/2012-8, all other named authors are partially supported by
CAPES, CNPq and FAPESP}
\author{J.-P. Brasselet}
\address{I.M.L.,  Aix-Marseille University, France.}
%\thanks{
%%
\author{A. K. M. Libardi}
\address{IGCE-UNESP, Rio Claro, S.P. Brasil.}
%%\thanks{nao posso esquecer de colocar o tematico de Topologia}
%%
\author{E. C. Rizziolli}
\address{IGCE-UNESP, Rio Claro, S.P. Brasil.}
%\thanks{here}
%%
\author{M. J. Saia}
\address{ICMC-USP, S\~ao Carlos, S.P. Brasil.}
%\thanks{here}


\begin{abstract} In this article we study  the bordism groups of normally nonsingular  maps $f: X \to Y$ defined on pseudomanifolds  $X$ and $Y$.  To characterize 
{the bordism of such maps, inspired by  the formula} 
given by Stong,   we  give a general definition of Stiefel-Whitney numbers defined on $X$ and $Y$ using the Wu classes defined by Goresky and Pardon in \cite{Gor-Pardon} and 
{we show} that in several cases the cobordism class of 
{a normally} nonsingular map   $f:X \to Y$  guarantees that these numbers are zero.
\end{abstract}

\maketitle \thispagestyle{empty}

 \section{Introduction }
 The ambiental bordism of manifolds was presented by  Thom in \cite{Thom}.
 Conner and Floyd  \cite{Conner} extended this theory to bordism of maps
 between closed manifolds and there is a classical work of Stong  \cite{Stong},
where the bordism class of maps between manifolds
 $f: X \to Y$ is characterized
 in terms of so-called  Stiefel-Whitney numbers of $(f,X,Y)$.

 {Concerning  the bordism} on singular varieties, Siegel in \cite{Siegel}  computed the bordism groups
 of $\mathbb Q$-Witt spaces, showing that in non trivial cases they are equal to the
 Witt groups.
 %$W({\mathbb Q})$.
 Pardon  \cite{Pardon} computed the bordism
groups of the  ``Poincar\'e duality spaces" 
{defined}  by Goresky and Siegel in \cite{Gor-Siegel}.

In this article we extend this notion to bordism groups of  normally nonsingular maps $f:X \to Y$  between  pseudomanifolds. For closed smooth manifolds $X$ and $Y$ this definition becomes the Stong's
definition of cobordism of maps $(f,X,Y)$ given in \cite{Stong}.

To characterize 
{the bordism of such maps, inspired by  the formula} 
given by Stong,   we  give a general definition of Stiefel-Whitney numbers defined on $X$ and $Y$ using the Wu classes defined by Goresky and Pardon in \cite{Gor-Pardon} and 
{we show}  that, in several cases the cobordism class of the map $f$ guarantees that these numbers are zero.
{More precisely, we } 
show how to extend the result of Stong in the case of  normally nonsingular maps $f:X \to Y$ in the following situations:
 Firstly we consider
the case  $X$ is a locally orientable $\Z_2$-Witt space of pure dimension $a$ and
$Y $ an $b$-dimensional smooth manifold. Then we consider the case $X$ is an  $a$-dimensional smooth manifold and $Y$ a locally orientable $\Z_2$-Witt space of pure dimension $b$. To conclude,
we consider the general case  where $X$ and $Y$ are locally orientable $\Z_2$-Witt spaces.

\newpage

  \section{Intersection homology}

\subsection{Pseudomanifolds}
\begin{definition}
 A pseudomanifold (without boundary) of dimension $a$ is a compact space $X$ which  is the closure of
the union of  $(a-1)$-dimensional simplices in any triangulation of $ X$, and each $(a - 1)$ simplex is a face
of exactly two $a$-simplices.

 A pseudomanifold (with boundary) of dimension $a$ is a compact space $X$ which  is the closure of
the union of  $(a-1)$-dimensional simplices in any triangulation of $ X$, and each $(a - 1)$ simplex is a face
of either one or two $a$-simplices. The boundary consists of simplices which are faces of only one  $a$-simplex.
\end{definition}

Every pseudomanifold admits a pieciewise linear (P.L. for short) stratification, which is a filtration by closed subspaces $\emptyset \subset X_0 \subset X_1 \subset  \ldots \subset X_{a-2}  \subset X_a = X$, such that for each point $x \in X_i-X_{i-1}$ there is a neighborhood $U$ and a P.L. stratum preserving homeomorphism between  $U$ and ${\mathbb R}^{a-i} \times  {\rm C}(L)$, where $L$ is the link
of the stratum $X_i-X_{i-1}$ and {\rm C}$(L)$ denotes the cone on $L$. Thus, if $X_i -X_{i-1}$ is non empty, it is a
(non necessarily connected) manifold of dimension $i$, and is called the
$i$-dimensional stratum of the stratification.

The singular part,  denoted by $\Sigma X$,  is contained in the element $X_{a-2} $ of the filtration.

%{\bf ???  A pseudomanifold $X$ is orientable if its nonsingular part  $X-\Sigma(X)$ is an orientable manifold.}

 %{ \bf \begin{definition}
 %An orientable  $n$-dimensional pseudomanifold $X$  is normal if $H_n(X,X-x; \mathbb Z) \simeq {\mathbb Z}$ for any $x %\in X$.
%\end{definition}

%\begin{definition}
 %A normalization of an $n$-dimensional pseudomanifold $X$  is a normal pseudomanifold $\tilde X$ together with a %finite projection $\pi : {\tilde X} \to X$ such that for any
 %$p\in X$, $$\pi_{\ast} : \bigoplus_{q \in \pi^{-1}(p)} H_n({\tilde X},{\tilde X}-q;\mathbb Z) \to H_n(X,X-p;\mathbb %Z)$$ is an isomorphism.
% \end{definition} ???}

\begin{definition} \cite{FM} A map  $f:X \to Y$ between   pseudomanifolds 
%of dimensions $a$ and $b$ respectively
 is normally nonsingular   if there exists a diagram
$$
\xymatrix{
 N \ar[d]_{\pi} \ar@{^{(}->}[r]^i  & Y \times \Bbb R^k \ar[d]^{p}\\
 X \ar@/_/[u]_{s} \ar[r]^{f}& Y   },$$
where $\pi:N \longrightarrow X$ is a vector bundle with zero-section $s$, $i$ is an open embedding,
$p$ is the   first
projection and one has $f = p \circ i \circ s$.  The bundle  $N = N_f$ is called the normal bundle.
\end{definition}

 \subsection{Intersection Homology and Cohomology}

All homology and cohomology groups will be considered with ${\mathbb Z}_2$ coefficients. Reference for this section is Goresky-MacPherson original paper \cite{GM1}.

The notion of perversity is fundamental for the definition of  intersection homology and cohomology. 
 A perversity $\overline p$ is a multi-index sequence of integers $(p(2),p(3), \ldots )$ such that $p(2)=0$ and $p(c) \leq p(c+1) \leq p(c)+1$, for $c\ge 2$. Any perversity $\overline p$ lies
between the zero perversity  ${\overline 0}=(0,0,0 , \ldots )$ and the total perversity ${\overline t}=(0,1,2,3,\ldots)$. In particular,
we will use the lower middle perversity, denoted ${\overline m}$ and the upper middle perversity, denoted ${\overline n}$,  such that
$${\overline m}(c) = \left[ \frac{c-2}{2} \right] \quad \text{ and } \quad {\overline n}(c) = \left[ \frac{c-1}{2} \right],
\qquad \text{ for } c\ge 2.$$

Let $X$ be an $a$-dimensional
pseudomanifold  and $\overline p$ a perversity.
The intersection homology groups with ${\mathbb Z}_2$ coefficients, denoted  $IH_i^{\overline p}(X)$,
are the homology groups of the chain complex
$$IC_i^{\overline p}(X)=\left \{\xi \in C_i(X) \ | \  \begin{array}{ccc}
  \dim (|\xi | \cap X_{a-c}) & \leq & i-c+p(c) \ {\rm and } \\
   \dim (|\partial \xi | \cap X_{a-c}) & \leq & i-1 -c+p(c)
\end{array} \right \}, $$ where $C_i(X)$ denotes the group of   compact $i$-dimensional P.L. chains
$\xi$ of $X$ with ${\mathbb Z}_2$ coefficients and $\vert \xi \vert$ denotes the support of $\xi$.

In fact $C_{\ast}(X)$ is the direct limit $\displaystyle \lim_\to C_{\ast}^{\mathcal T}(X)$, where $C_{\ast}^{\mathcal T}(X)$ is the simplicial chain complex with respect to a
triangulation $\mathcal T$ and the direct limit is taken with respect to subdivision within the family of triangulations of $X$ compatible with the filtration of $X$.

The intersection cohomology groups with ${\mathbb Z}_2$ coefficients, denoted $IH^{a-i}_{\overline p}(X)$, are defined as the groups of the cochain complex
$$IC^{a-i}_{\overline p} (X) = \left \{\gamma \in C^{a-i}(X) \ | \  \begin{array}{ccc}
  \dim (|\gamma | \cap X_{a-c}) & \leq & i-c+p(c) \ {\rm and } \\
   \dim (|\partial \gamma | \cap X_{a-c}) & \leq & i-1 -c+p(c)
\end{array} \right \}, $$ where $C^{a-i}(X)$ denotes the abelian group, with ${\mathbb Z}_2$ coefficients, of all
$(a-i)$-dimensional P.L. cochains of $X$ with closed supports in $X$.

The main properties of intersection homology that we will use are the following:

For any perversity $\bar{p}$, the Poincar\' e map PD, cap-product by the fundamental class
of $X$  naturally  factorizes in the following way
\cite{GM1}:
$$\xymatrix{
H^{a-i}(X) \ar[rr]^{ PD } \ar[rd]^{\alpha_X}
&& H_{i} (X) \\
&IH_{i}^{\bar{p}}(X). \ar[ru]^{\omega_X}&
  }$$
  where $\alpha_X$ is induced by the cap-product by the fundamental class $[X]$ and 
  $\omega_X$ is induced by the inclusion $IC_i^{\overline p}(X) \hookrightarrow C_i(X)$.

%We remember that the cohomology roups $H^{a-i}(X)$ and the homology groups $H_{i} (X)$ 
% are  considered with$\mathbb{Z}_2$ coefficients.

For perversities $\bar{p}$ and $\bar{r}$ such that $\bar{p}+\bar{r} \le \bar{t}$, the intersection product
$$IH_{i}^{\bar{p}}(X) \times IH_{j}^{\bar{r}}(X) \to
IH_{(i+j)-a}^{\bar{p}+\bar{r}}(X)
$$
is well defined.

%have natural  homomorphisms $\alpha_X : H^{i}(X;\mathbb{Z}_2) \to IH^{i}_{\bar{p}}(X;\mathbb{Z}_2)$
%and $\textcolor{red}{\omega_{X}}   : IH_{i}^{\bar{p}}(X;\mathbb{Z}_2) \to H_{i}(X;\mathbb{Z}_2) $.
%If $\bar{p}+\bar{r} \leq \bar{t}, \ (\bar{t}(c)=c-2)$ then there are cup products
%$$IH^{i}_{\bar{p}}(X;\mathbb{Z}_2) \times IH^{j}_{\bar{r}}(X;\mathbb{Z}_2) \to
% IH^{i+j}_{\bar{p}+\bar{r}}(X;\mathbb{Z}_2)$$ and cap products
%$$IH^{i}_{\bar{p}}(X;\mathbb{Z}_2) \times IH_{j}^{\bar{r}}(X;\mathbb{Z}_2) \to
%IH_{j-i}^{\bar{p}+\bar{r}}(X;\mathbb{Z}_2).$$

The natural homomorphism $IH^{a-i}_{\bar{p}}(X) \to IH_{i}^{\bar{p}}(X)$, cap-product by the fundamental class $[X]$,  is an isomorphism.

\section{Witt spaces and Wu classes}
In this section we use definitions and notations of M. Goresky  \cite{Goresky} and M. Goresky and W. Pardon \cite{Gor-Pardon}. First of all, let us 
{fix  notations} in the smooth case.

Let $X$ be an $a$-dimensional manifold. We will denote by $w^i (X) \in H^i (X)$
the Stiefel-Whitney
cohomology classes  (S-W cohomology classes) of the tangent bundle $TX$. The Stiefel-Whitney
homology classes  (S-W homology classes) of $TX$ denoted by $w_{a-i} (X) \in H_{a-i} (X)$  are their images by Poincar\'e duality. Let $i : X \hookrightarrow V$ be the inclusion of
differentiable manifolds, then one has the naturality formula $i^*(w^i (V)) = w^i (X)$.
\medskip

In the singular case, the {\it Steenrod square operations} are defined in intersection cohomology by M. Goresky {\cite[\S 3.4]{Goresky} } as follows:

\begin{definition} Let $X$ be an $a$-dimensional pseudomanifold. Suppose   $\bar c$ and $\bar d$ are perversities
such that $2{\bar c} \leq {\bar d}$.
For any $i$ with $ 0 \leq i \leq [a/2]$ the ``Steenrod square'' operation
$$Sq^i \colon IH^{j}_{\bar{c}}(X) \to IH^{i+j}_{{\bar{d}}}(X) \to {\mathbb Z}_2$$ is given 
by multiplication with the intersection cohomology $i^{th}$-Wu class of $X$:
$$v^i(X)= v^i_{\bar{d}}(X) \in IH^{i}_{\bar{d}}(X).$$
One defines $v^i(X)=0$, for $i > [a/2]$.
\end{definition}

\begin{definition} (\cite{Gor-Pardon}, Definition 10.1)
A stratified pseudomanifold $X$ is a ${\mathbb Z}_2$-Witt space if for each stratum of odd codimension $2k+1$,  $IH_{k}^{\bar{n}}(L)=0$, where $L$
is the link of the stratum.
\end{definition}

For such spaces, the middle intersection homology group satisfies the Poincar\'e duality over ${\mathbb Z}_2$.

In the following, we will use the notion of {\it locally orientable Witt-space} that we recall:
\begin{definition} (\cite{Gor-Pardon}, Definition 10.2)
A stratified pseudomanifold $X$ is a locally orientable Witt space if it is both locally orientable
and a $\mathbb{Z}_2$-Witt space.
\end{definition}

Let  $X$ be a ${\mathbb Z}_2$-Witt space, then the Wu classes $v^i(X)$ lift canonically to $IH^{i}_{\bar{m}}(X)
= IH^{i}_{\bar{n}}(X) $ (see  \cite{Gor-Pardon} \S 10).
We denote by  $v_{a-i}(X) \in IH^{\bar n}_{a-i}(X)$  the
(homology) $(a-i)^{th}$-Wu class of $X$,
in intersection homology,   dual to the Wu class $v^{i}(X)$
(denoted by $Iv^i \in IH^i(X)$ in \cite{Goresky}).

\begin{definition} \cite{Goresky,Siegel}
One defines the Whitney classes by
$$ IW_{a-i} (X)= \sum_{\ell+j=i} Sq^{\ell}v^{j}(X) \in IH^i_{\bar t} (X) =
H_{a-i}(X).$$
\end{definition}

The pullback of the intersection cohomology Whitney class under a normally nonsingular map
is given by the following theorem (\cite[5.3]{Goresky}):
\begin{theorem}
Let $X$ and $Y$ be ${\mathbb Z}_2$-Witt spaces and $f:X \to Y$  a normally nonsingular map with normal bundle $N_f$. Then one has, in $IH^*(X)$:
$$f^* (IW(Y)) = W(N_f) \cup IW(X)$$
where $W(N_f) $ is the Whitney cohomology class (in $H^*(X)$) of the normal bundle $N_f$.
\end{theorem}

The inclusion $j: X \hookrightarrow V$ provides an unique morphism
$ j^* : IH_{\bar q}^{n-i}(V) \to  IH_{\bar q}^{n-i}(X)$ (see \cite{BBFGK} \S (3.4)).
The result comes from the commutative diagram
$$\xymatrix{
  IH_{\bar q}^{n-i}(X) &  IH_{\bar q}^{n-i}(V) \ar[l]^{j^*} \\
IH_{i}^{\bar{p}}(X)   \ar[u]_{\cong} & IH_{i+1}^{\bar{p}}(V)  \ar[l]^{j^*_X}  \ar[u]^{\cong} }$$
where the bottom map $j^*_X$ is defined by the upper one. We have:

\begin{corollary}\label{equ3} 
Let us consider the inclusion $j: X \hookrightarrow V$ of the ${\mathbb Z}_2$-Witt space  $X$ in a
${\mathbb Z}_2$-Witt space $V$ such that $\Sigma V \subset X$, 
so that the normal bundle $N_i$
is trivial. Then one has:
$$
j^*_X (v_{i+1}(V)) = v_i (X).
$$
\end{corollary}

\section{Cobordism of maps}

\begin{definition} \label{triomap}

Let  $f:X \to Y$ be a normally nonsingular map between   pseudomanifolds of dimensions $a$ and $b$ respectively. The triple $(f,X,Y)$ bords  if there exist:

%, where $X$ and $Y$ are Whitney stratified%
\begin{enumerate}

\item Pseudomanifolds $V$ and $W$ with dimensions $a+1$ and $b+1$,
respectively, such that  $\partial V= X$ and $\partial W =Y$; $\Sigma V \subset X$ and $\Sigma W \subset Y$.

\item $F: V \to W$ normally nonsingular such that ${F \mid}_X = f$.

\end{enumerate}

We will denote $(f,X, Y) = \partial(F,V,W)$.
\end{definition}

The definition implies that $V \setminus X$ and $W \setminus Y$ are smooth manifolds. If $X$ (resp. $Y$) is
a manifold, then $W$ (resp. $V$) is a manifold with smooth boundary.

%The set of equivalence classes under this relation will be denoted by ${\mathcal N}(m,n)$.

If we consider $X$ and $Y$ closed smooth manifolds, this definition becomes the Stong's definition to cobordism of maps $(f,X,Y)$ in \cite{Stong}. In this case Stong defines the Stieffel-Whitney (S-W for short) numbers   associated to the map $(f,X,Y)$;  these numbers allow to characterize the bordism properties among such maps.  We recover here results described by Stong which are necessary to better understand our main results.

\begin{definition} \cite{Stong}
Let us consider a map $f: X \to Y$, where $X$ and $Y$ are manifolds of dimensions $a$ and $b$, respectively. Define $f^{!}: H^i(X) \to H^{i+b-a}(Y)$ in such a way that for any $\alpha \in H^{i}(X)$,
we define $f^{!}(\alpha): H_{i+b-a}(Y) \to {\mathbb Z}_2$ such that for each
$\beta \in  H_{i+b-a}(Y)$, $$f^{!}(\alpha)(\beta)=\langle f^{\ast}(\tilde{\beta}) \cup \alpha, [X]\rangle \in {\mathbb Z}_2,$$
 where $\tilde{\beta} \in H^{a-i}(Y)$ is the Poincar\'e dual of $\beta$.
\end{definition}

\begin{remark} \label{Atiyah-Hirzebruch} According to Atiyah and Hirzebruch \cite{Atiyah-Hir}, the map $f{^!}$ can be described  in the following way: let us consider $h: X \to {\mathbb S}^s$ an imbedding of $X$ in some $s$-dimensional sphere ${\mathbb S}^s$
{ and $T$  a tubular} neighborhood of $(f\times h)(X)$ in $Y \times {\mathbb S}^s$, then $f^{!}$ is the composition of the maps:
$$ H^i(X) \stackrel{\varphi}{\to } H^{i+s+b-a}(T/ \partial T) \stackrel{c^{\ast}}{\to }H^{i+s+b-a}(Y\times {\mathbb S}^s) \simeq H^{i+b-a}(Y), $$ where $\varphi$ denotes the Thom isomorphism and $c:Y \times {\mathbb S}^s \to T/ \partial T$ is the contraction.
\end{remark}

%\newpage

\section{Main results}

In this section we show how to extend the result in the case of singular spaces and
normally nonsingular
maps $f : X \longrightarrow Y $. Firstly we consider
the case  $X$ is a locally orientable $\Z_2$-Witt space of pure dimension $a$ and
$Y $ a $b$-dimensional smooth manifold. 
Then we consider the case $X$ is an  $a$-dimensional smooth manifold 
{ and $Y$ is a locally} orientable $\Z_2$-Witt space of pure dimension $b$. To conclude,
we consider the general case  where $X$ and $Y$ are locally orientable $\Z_2$-Witt spaces.

\subsection{Case of a map $f : X \longrightarrow Y $, with $Y$ a   smooth manifold.} $ $

Let $f : X \longrightarrow Y $ be a normally nonsingular map, with $X$ a locally orientable $\Z_2$-Witt space of pure dimension $a$ and $Y$ a $b$-dimensional smooth manifold.

\begin{definition}
Let us define the map $f_B:IH_{i}^{\bar{^p}}(X) \to IH_{i}^{\bar{^p}}(Y)$
in such a way that the following diagram commutes
$$\xymatrix{
   H_{i}(X) \ar[r]^{  \ f_{\ast}} & H_{i}(Y) \\
   IH_{i}^{\bar{p}}(X) \ar[u]_{\omega_X} \ar [r]^{f_B}& IH_i^{\bar{p}}(Y) \ar[u]^{\omega_Y}_\simeq
   }$$
   %{\stackrel{\omega_N}}_{\simeq}}$$
{\it i.e.} $f_B = (\omega_{{Y}})^{-1} \circ f_\ast \circ \omega_{X},$ where the map  $\omega_{{Y}}$
is an isomorphism  since $Y$ is smooth.
\end{definition}

We denote by $\widetilde f_B$ the map obtained by composition
$$IH_{i}^{\bar{p}}(X)  \buildrel{\omega_{X}}\over\longrightarrow H_{i}(X)
 \buildrel{f_*}\over\longrightarrow H_{i}(Y)
  \buildrel{PD}\over\longrightarrow H^{b-i}(Y)
$$
with Poincar\'e duality $PD$.

\begin{definition}
For any partition   $\iota = (\iota_1, \ldots, \iota_s)$
and $r$ numbers $u_1, \ldots , u_r$ satisfying
\begin{equation} \label{equ1}
(\iota_1+ \cdots + \iota_s) +u_1+ \cdots + u_r +r(b-a)=b,
\end{equation}
let us denote $ w^{\iota}(Y) =  w^{\iota_1 }(Y) \cdots  w^{\iota_s}(Y)$. 
  The S-W numbers of any triple $(f,X,Y)$ are defined by
  $$ \langle w^{\iota}(Y).{\widetilde f_{B}}(v_{a-{u}_{1}}(X)). \cdots .
  {\widetilde f_{B}}(v_{a-{u}_{r}}(X)),[Y]\rangle. $$
\end{definition}

\begin{theorem} \label{TeoXN}  Let $f : X  \longrightarrow Y  $ be a normally nonsingular map, with $X $ a locally orientable $\Z_2$-Witt space of pure dimension $a$ and
$Y $ a
$b$-dimensional smooth manifold. If $(f,X,Y)$ bords, then for any  partition  $\iota$ and $r$ numbers $u_1, \ldots , u_r$ satisfying (\ref{equ1}), the S-W numbers  $$ \langle w^{\iota}(Y).{\widetilde f_{B}}(v_{a-{u}_{1}}(X)). \cdots .
  {\widetilde f_{B}}(v_{a-{u}_{r}}(X)),[Y]\rangle$$ are zero.
\end{theorem}

\begin{proof} As $(f,X,Y)$ bords, one has $(f,X, Y) = \partial(F,V,W)$. We may   define a map
 $${\widetilde F_B} : IH_{i}^{\bar{p}}(V) \to IH_{i}^{\bar{p}}(W) = H_{i}(W) \to H^{b+1-i}(W)$$
 in the same way that we defined ${\widetilde f_B}$.

 %and by remarking that we also have the long exact sequence in intersection homology groups,
One has:
$$ \langle w^{\iota}(Y).{\widetilde f_{B}}(v_{a-{u}_{1}}(X)). \cdots .
  {\widetilde f_{B}}(v_{a-{u}_{r}}(X)),[Y]\rangle =$$
$$\langle j^{\ast} w^{\iota}(W).j^{\ast}
{\widetilde F_{B}}(v_{a-{u}_{1}}(V)). \cdots . j^{\ast}{\widetilde F_{B}}(v_{a-{u}_{r}}(V)),\partial [W]\rangle,$$
by corollary \ref{equ3} and commutativity of the following diagram:

$$\xymatrix{
   H_{i+1}(V) \ar[rrrr]^{F_*} \ar[rd]^{\tilde j} &&&& H_{i+1}(W)\ar[ld]_{\tilde j}\\
   & H_i(X) \ar[rr]^{f_*} && H_i(Y) \\
   & IH_{i}^{\bar{p}}(X) \ar[u]^{\omega_{X}}  \ar[rr] ^{{\widetilde{f_B}}}&& H^{b-i}(Y) \ar[u]^{PD}_{\cong}\\
    IH_{i+1}(V) \ar[rrrr]^{{\widetilde F_{B}}}  \ar[ru]^{j^*_X}   \ar[uuu]^{\omega_{V}} &&&& H^{b-i}(W).
    \ar[lu]_{j^*} \ar[uuu]^{PD}_{\cong}
 }$$

So, we obtain:
$$\left\langle j^{\ast} \left ( w^{\iota}(W).{\widetilde F_{B}}(v_{a-{u}_{1}}(V)). \cdots .{\widetilde F_{B}}(v_{a-{u}_{r}}(V))\right ),\partial [W]\right\rangle =$$

$$\left\langle \delta j^{\ast} \left ( w^{\iota}(W).{\widetilde F_{B}}(v_{a-{u}_{1}}(V)). \cdots . {\widetilde F_{B}}(v_{a-{u}_{r}}(V))\right ), [W, \partial W]\right\rangle =0,$$
where
$$ H^k(W) \stackrel{j^{\ast} }{\to } H^{k}(Y) \stackrel{\delta}{\to }H^{k+1}(W, \partial W)$$
is part of a long exact sequence, so that $\delta j^{\ast}  = 0$.
\end{proof}

\subsection{Case of a map  $f : X  \longrightarrow Y  $, with $X$ a   smooth manifold} $ $

Let $f : X  \longrightarrow Y  $ be a normally nonsingular map, with $X $ an  $a$-dimensional smooth manifold and $Y $ a locally orientable $\Z_2$-Witt space of pure dimension $b$.

Since $f$ is a normally nonsingular map one may consider  the  normal bundle $N_f$ over $X$, 
and   $i:N_f \to Y \times {\mathbb R}^{s+1}$ an open imbedding. Let  $\widetilde T$ be
 a tubular neighborhood of $(f\times h)(X)$ in $Y \times {\mathbb R}^{s+1}$, where $h: X \to {\mathbb R}^{s+1}$ is defined in such a way that  the following  diagram  commutes.
$$\xymatrix{
  N_f  \ar[r]^{i}  & Y \times {\mathbb R}^{s+1} \\
   X \ar[u]_{\sigma} \ar[ru]_{  f \times h} & } $$

 \bigskip
We denote by  ${\mathbb S}^s$  the $s$-dimensional sphere in ${\mathbb R}^{s+1}$ and by  $T$  the 
intersection  $T = \widetilde T \cap ( Y \times {\mathbb S}^{s})$. 

Following the remark \ref{Atiyah-Hirzebruch}, there exists a map $\phi$ which is the composition of the maps: $$ H^i(X) \stackrel{\varphi}{\to } H^{i+s+b-a}(T/ \partial T) \stackrel{c^{\ast}}{\to }H^{i+s+b-a}(Y\times {\mathbb S}^s) \simeq H^{i+b-a}(Y), $$ here $\varphi$ denotes the Thom homomorphism and $c:Y \times {\mathbb S}^s \to T/ \partial T$ is the contraction. The last isomorphism is given by the $K\ddot{u}nneth formula$ for a product of a smooth manifold with a $Z_2$- Witt space \cite{BBFGK}.

Since $X$ is a smooth manifold, $\alpha_{X}: H^i(X) \to IH^{\bar{p}}_{a-i}(X)$ is an isomorphism, then
one defines the map $f_B$ by commutativity of the following diagram, {\it i.e.}  as being
$f_B = \alpha_{Y} \circ \phi \circ \alpha_{X}^{-1}$
 $$\xymatrix{
H^i(X) \ar[r]^{ \phi} \ar[d]_{\alpha_{X}}^{\cong}       & H^{b-(a-i)}(Y) \ar[d]^{\alpha_Y} \\
 IH^{\bar{p}}_{a-i}(X)  \ar [r]^{f_B}& IH_{a-i}^{\bar{p}}(Y).  }$$
 For any $u$ with $0 \le  u \le b$,
 let $v_{u}(Y) \in IH_{u}^{\bar{m}}(Y)$ the Wu class of $Y$, dual of
 $v^{b-u}(Y) \in IH^{b-u}_{\bar{m}}(Y)$ and $w_{b-u}(X)$
 the homology Whitney class of $X$, so that $f_B (w_{b-u}(X)) \in IH_{b-u}^{\bar{m}}(Y)$.
 For any $u$ with $0 \le  u \le b$ the S-W intersection numbers
  $$v_{u}(Y) \, . \, {f_{B}}(w_{b-u}(X))$$ are well defined.

\begin{theorem} \label{TeoMY}
Let $f : X  \longrightarrow Y  $ be a normally nonsingular map, with $X $ an  $a$-dimensional smooth manifold and $Y $ a locally orientable $\Z_2$-Witt space of pure dimension $b$.
 If $(f,X,Y)$ bords, then for any  $0 \le  u \le b$  the S-W numbers
 $$v_{u}(Y) \, . \, {f_{B}}(w_{b-u}(X))$$ are zero.
\end{theorem}

\begin{proof}
  If $(f,X,Y)=\partial(F,V,W)$, one has  $$ H^i(V) \stackrel{\varphi}{\to } H^{i+s+b-a}(T'/ \partial T') \stackrel{c^{\ast}}{\to }H^{i+s+b-a}(W\times {\mathbb S}^s) \simeq H^{i+b-a}(W), $$  where $V$ is embedded in ${\mathbb S}^s$ and $T'$ is a tubular neighborhood of $(F \times h)(V)$, which gives rise to the corresponding map $F_B$. Therefore  we can consider the following diagram, where PD denotes the Poincar\'e duality

$$\xymatrix{
  & H_{a-i}(X) \ar[r]^{  \ f_{\ast}} & H_{a-i}(Y) &\\
   & IH_{a-i}^{\bar{p}}(X) \ar[u]_{\omega_{X}} \ar[r]^{f_B}& IH_{a-i}^{\bar{p}}(Y) \ar[u]^{\omega_{Y}} & \\
 \ar[ruu]^-{\stackrel{PD}{\simeq}} H^{i}(X) \ar[r]^-{\varphi} \ar[ru]^{ \ \ \ \alpha_X}  \ar[r] & H^{i+s+b-a}(T/\partial T) \ar[rr]^-{c^{\ast}} & & H^{i+s+b-a}(Y \times {\mathbb S}^s) \simeq H^{i+b-a}(Y) \ar[lu]_{\alpha_Y}  \ar[luu]_-{PD}}$$
 since we had defined $f_B$ and $F_B$ the result follows in the same way of the proof of  Theorem \ref{TeoXN}.
\end{proof}

\subsection{The general case.} $ $

In the general case  $X$ and $Y$ are locally orientable $\Z_2$-Witt spaces of dimensions $a$ and $b$ respectively. It is  not always possible to define an unique  map $f_B$  as done in the other cases, however we can show that for any map  $\tilde{f}_B$ considered, the bordism condition of $(f,X,Y)$  implies that the corresponding S-W numbers are zero.

First we show the following lemma.

\begin{lemma} \label{lema5.6} Let $f : X \longrightarrow Y $ be a normally nonsingular map and $(f,X, Y) = \partial(F,V,W)$. Given a map $f_B$ there exists a map $F_B$ such that the following diagram  commutes.

$$\xymatrix{
   IH^{\bar{m}}_{u}(X) \ar[d]_{\ f_{B}}  & IH^{\bar{m}}_{u+1}(V)  \ar[l]_{ j^{\ast}_X} \ar[d]_{F_B} \\
   IH^{\bar{m}}_{u}(Y)  & IH^{\bar{m}}_{u+1}(W).  \ar[l]_{j^{\ast}_Y} }$$
\end{lemma}

\begin{proof} The diagram $$\xymatrix{
  X \ar [d]_{f} \ar @{^{(}->}  [rr] ^{j_X} && V  \ar[d]_{F}  \\
  Y   \ar @{^{(}->}  [rr] ^{j_Y}
 && W
  } $$ is a cartesian diagram. Then we can apply Proposition $10.7$ in \cite{Borel} (see also \cite{GM2}).

 One has equality of sheaves on $Y$: $$j^{\ast}_Y F_{!} \mathcal{A} =f_{!}j^{\ast}_X \mathcal{A}$$
 for any sheaf $\mathcal{A}$ on $V$. That provides
a commutative diagram of complexes of sheaves on $Y$
 {(perverse intersection sheaves for the middle perversity $\bar m$)}.

$$\xymatrix{
  f_{!}{\mathcal IC}^\bullet_X  \ar[d]_{}  \ar @{<-}[r]^-{ }
  & f_{!}j^{\ast}_{X} \left( {\mathcal IC}^\bullet_V \right)=j^{\ast}_{Y}F_{!}\left( {\mathcal IC}^\bullet_V \right) \ar[d]_{}  \\
 {{\mathcal IC}^\bullet_Y } \ar @{<-} [r]^-{} &  j^{\ast}_Y\left ({\mathcal IC}^\bullet_W \right )}$$

Let us remind that intersection homology is obtained by taking hypercohomology of the perverse intersection sheaf: $$ IH^{\bar m}_u(Y)=I\!\!H^{b-u}(Y;  {{\mathcal IC}^\bullet_Y } )$$

\medskip

Taking hypercohomology $$I\!\!H^{b-u}(Y;\bullet)$$ of the previous diagram, one obtains:

$$\xymatrix{
  I\!\!H^{b-u}( X; {{\mathcal IC}^\bullet_X } )  \ar[d]_{f_B} \ar @{<-}[r]^-{ j_{X}^{\ast}}
  & I\!\!H^{b-u}( V; {{\mathcal IC}^\bullet_V }) \ar[d]_{F_B} \\
I\!\!H^{b-u}(Y; {{\mathcal IC}^\bullet_Y } ) \ar @{<-} [r]^-{j_Y^{\ast}}& I\!\!H^{b-u}( W; {{\mathcal IC}^\bullet_W }) }$$
 and the Lemma follows.

\end{proof}
%%%%%%%%%%%%%%

\begin{theorem} \label{TeoXY}  Let $f : X \longrightarrow Y $ be a normally nonsingular map, with $X$ and $Y$ locally orientable $\Z_2$-Witt spaces of pure dimension $a$  and $b$ respectively. Then for any $u$ with $0 \le u \le b$, the S-W numbers $\displaystyle \langle v_{u}(Y).f_{B}(v_{b-u}(X)),[Y]\rangle$ are zero.

\end{theorem}

\begin{proof}
The diagram of Lemma \ref{lema5.6} can be written in the cohomology setting

$$\xymatrix{
   I  H_{\bar{n}}^{a-u}(X) \ar[r]^{  \ f^{B}} & I  H_{\bar{n}}^{b-u}(Y)  \\
   I  H_{\bar{n}}^{a-u}(V)  \ar[u]_{j^{\ast}_X} \ar[r]^{F^B}& I H_{\bar{n}}^{b-u}(W). \ar[u]_{j^{\ast}_Y} } $$
where $\bar{m}+\bar{n}=\bar{t}$ and we use the same notation for corresponding maps ${j^{\ast}_X}$ and ${j^{\ast}_Y}$.

Let us consider the homology class $v_{b-u}(Y) \in I  H_{b-u}^{\bar n}(Y)$, that will be written $v^u(Y) \in I  H^{u}_{\bar m}(Y)$ in the cohomology setting.

Then $v^u(Y)= j^{\ast}_{Y}v^u(W)$ where $v^u(W) \in I  H^{u}_{\bar m}(W)$ is the corresponding cohomology Wu class  to the homology Wu class  $v_{b+1-u}(W) \in I  H^{\bar n}_{b+1-u}(W)$
of $W$.

\medskip
Let us consider the cohomology Wu class $v^{a-u}(X) \in I H ^{a-u}_{\bar n}(X)$ corresponding for the cohomology {Wu class} $v_{u}(X) \in I H^{\bar m}_{u}(X)$. 
Then  {$v^{a-u}(X)=j^{\ast}_X((v^{a-u}(V))$ }
where $v^{a-u}(V) \in I H^{a-u}_{\bar n}(V)$ is the corresponding cohomology Wu class  to the homology class $v_{u+1}(V) \in I H^{\bar m}_{u+1}(V)$.

One has
$$ f^B(v^{a-u}(X))=f^Bj^{\ast}_X(v^{a-u}(X))\in I H^{b-u}_{\bar n}(Y).$$

The intersection product $$ v_{b-u}(Y)\cdot f_B(v_u(X)) \in IH_{b-u}^{\bar n}(Y)\times I H^{\bar m}_{u}(Y) \to I H^{\bar t}_0(Y)$$
corresponds to the product $$ v^u(Y) \cup f^B(v^{a-u}(X)) \in I H^{u}_{\bar m}(Y) \times I H^{b-u}_{\bar n} (Y) \to I H^{b}_{\bar 0}(Y).$$

One has

 $$\langle v^{u}(Y)\cup f^{B}(v^{a-u}(X)), [Y]\rangle =$$

 $$\langle j^{\ast}_Y v^{u}(W) \cup f^Bj^{\ast}_X(v^{a-u}(V),[Y]\rangle = $$

$$\langle j^{\ast}_Y v^{u}(W) \cup j^{\ast}_YF^{B}(v^{a-u}(V)),[Y]\rangle =$$

$$\langle j^{\ast}_Y\left[v^{u}(W)\cup F^{B}(v^{a-u}(V))\right], [\partial W]\rangle =$$

$$\langle \delta j^{\ast}_Y \left[v^{u}(W) \cup F^{B}(v^{a-u}(V))\right], [W,\partial W]\rangle =0$$
where the first equality is a consequence of the Theorem 5.3 of Goresky \cite{Goresky}, the second one is from Lemma \ref{lema5.6} and the fourth equality is obtained in an analogous way than the proof of Theorem \ref{TeoMY}.
\end{proof}


\begin{thebibliography}{G}

\bibitem{Atiyah-Hir} M. F. Atiyah and F. Hirzebruch,  {\em Cohomologie-Operationen und charaktieristische Klassen,} Maht. Z. {\bf 77} (1961), 149--187.

\bibitem {Borel} A. Borel et al. Intersection Cohomology,  {\em Progress in Mathematics},  {\bf 50},  Swiss seminars , Boston, 1984, ISBN: 0-8176-3274-3.

 \bibitem{BBFGK}  G. Barthel, J.-P. Brasselet, K. Fieseler, O. Gabber  and L. Kaup,
 {\em Rel\`evement de cycles alg\'ebriques et homomorphismes associ\'es en homologie d'intersection [Lifting of algebraic cycles and associated homomorphisms in intersection homology]}, Ann. of Math. {\bf  141}, no 2, (1995), 147--179.

\bibitem {Conner} P. E. Conner and  E. E. Floyd, {\em Differentiable Periodic Maps}, Springer, Berlin, 1964.

\bibitem{FM} W. Fulton and R. MacPherson, {\em Categorical framework for the study of singular spaces}. Mem. Amer. Math. Soc. {\bf 31}, no. 243, 1981.

\bibitem{Goresky} M. Goresky,   {\em Intersection Homology operations}, Comment. Math. Helvet. {\bf 59}, (1984), 485--505.

\bibitem{GM1} M. Goresky and R. MacPherson, {\em Intersection homology theory}, Topology {\bf 19} (1980), no. 2, 135--162.

\bibitem{GM2} M. Goresky and R. MacPherson, {\em Intersection homology II}, Invent. Math. {\bf 72} (1983), no. 1, 77--129.

\bibitem{Gor-Pardon} M. Goresky and W. Pardon,   {\em Wu Numbers of Singular Spaces}, Topology {\bf 28}, (1989), 325--367.

\bibitem{Gor-Siegel}  M. Goresky and P. Siegel,  {\em Linking pairings on Singular spaces}, Comment. Math. Helvet. {\bf 58}, (1983), 96--110.

\bibitem{Pardon} W. Pardon,   {\em Intersection homology Poincar\'e spaces and the characteristic variety theorem}, Comment. Math. Helvet. {\bf 65}, (1990), 198--233.

\bibitem{Siegel} P. Siegel,  {\em Witt Spaces: a Geometric Cycle Theory for KO-Homology at odd primes}, Amer. J. Math., {\bf 105}, (1983), no. 5, 1067--1105.

\bibitem{Stong} R. E. Stong, {\em Cobordism of maps}, Topology {\bf 5}, (1966), 245--258.

\bibitem {Thom} R. Thom, {\em Quelques propri\'et\'es globales des vari\'et\'es differentiables}, Comment. Math. Helvet. {\bf 28}, (1954), 17--86.

\end{thebibliography}
\end{document}